\documentclass[a4paper,12pt]{article}
\usepackage[top=2.5cm,bottom=2.5cm,left=2.5cm,right=2.5cm]{geometry}
\usepackage{cite, amsmath, amssymb}
\usepackage{amssymb}
\usepackage{graphicx}
\newenvironment{proof}{{\noindent \it Proof.}}{\hfill $\blacksquare$\par}
\usepackage{latexsym}
\usepackage{graphicx,booktabs,multirow}
\usepackage{tikz}
\usetikzlibrary{decorations.pathreplacing}
\usetikzlibrary{intersections}
\usepackage{enumerate}
\usepackage{graphicx,booktabs,multirow}
\usepackage{appendix}
\newtheorem{theorem}{Theorem}[section]

\newtheorem{lemma}[theorem]{Lemma}
\newtheorem{corollary}[theorem]{\rm\bfseries Corollary}

\usepackage[section]{placeins}
\allowdisplaybreaks

\begin{document}

\vspace*{10mm}

\noindent
{\Large \bf Proof of an open problem on the Sombor index}

\vspace{7mm}

\noindent
{\large \bf Hechao Liu}

\vspace{7mm}

\noindent
School of Mathematical Sciences, South China Normal University,  Guangzhou, 510631, P. R. China,
e-mail: {\tt hechaoliu@m.scnu.edu.cn} \\[2mm]

\vspace{7mm}

\noindent
Received 18 August 2022

\vspace{10mm}

\noindent
{\bf Abstract} \
The Sombor index is one of the geometry-based descriptors, which was defined as
$$SO(G)=\sum_{uv\in E(G)}\sqrt{d^{2}(u)+d^{2}(v)},$$
where $d(u)$ (resp. $d(v)$) denotes the degree of vertex $u$ (resp. $v$) in $G$.

In this note, we determine the maximum and minimum graphs with respect to the Sombor index among the set of graphs with vertex connectivity (resp. edge connectivity) at most $k$, which solves an open problem on the Sombor index proposed by Hayat and Rehman [On Sombor index of graphs with a given number of cut-vertices, {\it MATCH Commun. Math. Comput. Chem.\/} {\bf 89} (2023) 437--450]. For some of the conclusions of the above paper, we give some counterexamples.
At last, we give the QSPR analysis with regression modeling and Sombor index.

\vspace{5mm}



\baselineskip=0.30in

\section{Introduction}
\hskip 0.6cm
Let $G$ be a simple undirected connected graph with vertex set $V(G)$ and edge set $E(G)$.
Let $N_{G}(u)$ (or $N(u)$ for short) be the neighbor of vertex $u$ in $G$, then the degree of vertex $u$ is $d(u)=|N(u)|$.

Let $uv\notin E(G)$ (resp. $uv\in E(G)$), then $G+uv$ (resp. $G-uv$) denotes the graphs obtained from $G$ by adding (resp. deleting) the edge $uv$.
Denote by $C_{n}$, $K_{n}$, $S_{n}$, $P_{n}$, the cycle, complete graph, star graph, path with order $n$, respectively.
In this paper, all notations and terminologies used but not defined can refer to Bondy and Murty \cite{bond2008}.

One of the geometry-based indices, Sombor index \cite{gumn2021} was defined as
$$SO(G)=\sum_{uv\in E(G)}\sqrt{d^{2}(u)+d^{2}(v)},$$
where $d(u)$ (resp. $d(v)$) denotes the degree of vertex $u$ (resp. $v$) in $G$.
One can refer to \cite{rirm2021,chli2021,dengt2021,giva2021,guma2021,lizh2022} for more details about Sombor index.

We call a graph $G$ is $k$-connected if $G\cong K_{k+1}$, or $|V(G)|\geq k+2$ and $G$ has no $(k-1)$-vertex cut.
We all a graph $G$ is $k$-edge-connected $(k\geq 1)$ if $|V(G)|\geq 2$ and $G$ has no $(k-1)$-edge cut.
Denote by $\kappa(G)=\max\{k| \text{ $G$ is $k$-connected}\}$ the connectivity of connected graph $G$.
Denote by $\kappa'(G)=\max\{k| \text{ $G$ is $k$-edge-connected}\}$ the edge connectivity of connected graph $G$. Then $\kappa(G)\leq \kappa'(G)\leq n-1$ and $\kappa(G)=n-1\Leftrightarrow \kappa'(G)=n-1\Leftrightarrow G\cong K_{n}$.

Let $\mathbb{V}_{n}^{k}$ (resp. $\mathbb{E}_{n}^{k}$) be the set of graphs with $n$ vertices and $\kappa(G)\leq k\leq n-1$ (resp. $\kappa'(G)\leq k\leq n-1$).
Suppose $G_{1}$ and $G_{2}$ are two disjoint graphs. Let $G_{1}\vee G_{2}$ be the graph obtained from $G_{1}\cup G_{2}$ by adding edges between any vertex of $G_{1}$ and any vertex of $G_{2}$.
Let $K_{n}^{k}\triangleq K_{1}\vee K_{k}\vee K_{n-k-1}$.

The remainder of this paper is organized as follow. In Section 2, we determine the maximum and minimum graphs with respect to the Sombor index among the set of graphs with vertex connectivity (resp. edge connectivity) at most $k$.
In Section 3, we give some counterexamples about some results of \cite{hare2023}.
In Section 3, we give the QSPR analysis with regression modeling and Sombor index.

\section{Extremal graphs with connectivity at most $k$}
\hskip 0.6cm
In one recently published paper \cite{hare2023}, Hayat and Rehman proposed the open problem that considering the maximum and minimum graphs with respect to the Sombor index among the set of graphs with vertex connectivity (resp. edge connectivity) at most $k$. In the section, we completely solve the problem.

\begin{lemma}\label{l2-1}{\rm\cite{gumn2021}\rm}
Let $T$ be a tree with $n$ vertices. Then
$$
SO(P_{n})\leq SO(T)\leq SO(S_{n})
$$
with equality if and only if $T\cong P_{n}$ or $T\cong S_{n}$.
\end{lemma}

\begin{lemma}\label{l2-2}{\rm\cite{lyht2023}\rm}
Let $x>a\geq 1,y>0$, $f(x,y)=\sqrt{x^{2}+y^{2}}-\sqrt{(x-a)^{2}+y^{2}}$. Then
$f(x,y)$ is strictly increasing with $x$, strictly decreasing with $y$.
\end{lemma}

\begin{lemma}\label{l2-3}
If $u,v\in V(G)$ and $uv\notin E(G)$, then $SO(G)<SO(G+uv)$.
\end{lemma}

\begin{figure}[ht!]
  \centering
  \scalebox{.16}[.16]{\includegraphics{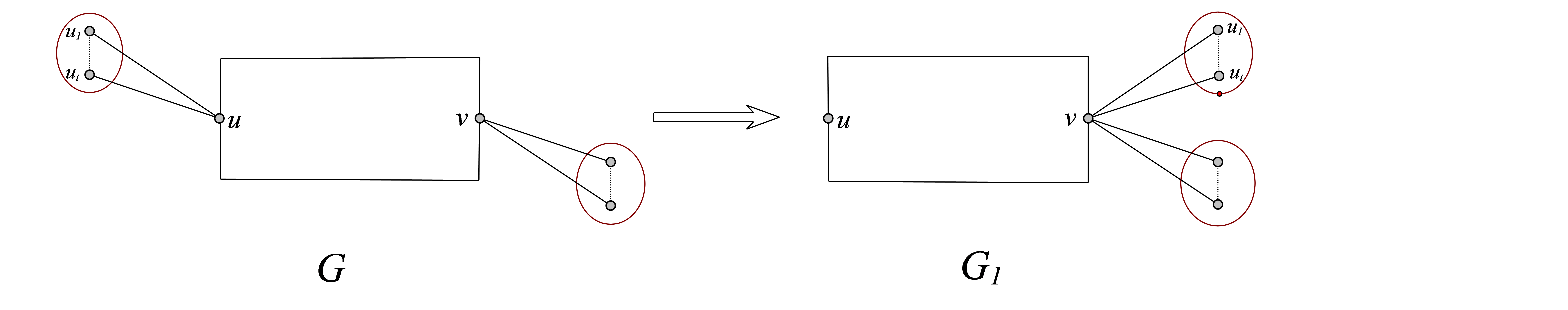}}
  \caption{The graphs $G$ and $G_{1}$ of Lemma \ref{l2-4}.}
 \label{fig-21}
\end{figure}

\begin{lemma}\label{l2-4}
Suppose $G$ is a connected graph, $u,v\in V(G)$, and $u_{1},u_{2},\cdots,u_{t}\in N_{G}(u)\setminus \{N_{G}(v)\cap N_{G}(u)\}$ for $1\leq t\leq d_{u}$ and $u_{1},u_{2},\cdots,u_{t}\notin P_{uv}$ where $P_{uv}$ is a path from $u$ to $v$. Let $G_{1}=G-\{uu_{1},uu_{2},\cdots,uu_{t}\}+\{vu_{1},vu_{2},\cdots,vu_{t}\}$ (see Figure \ref{fig-21}).
If $d_{u}\leq d_{v}$ and $uv\notin E(G)$, then $SO(G_{1})>SO(G)$.
\end{lemma}
\begin{proof}
For convenience, let $M=N_{G}(v)\cap N_{G}(u)$, $M_{1}=N_{G}(u)\setminus \{N_{G}(v)\cap N_{G}(u)\}$, $M_{2}=N_{G}(v)\setminus \{N_{G}(v)\cap N_{G}(u)\}$.
By the definition of Sombor index and the structure of graphs $G$ and $G_{1}$, we have
\begin{align*}
SO(G_{1})-SO(G)=&  \sum_{v_{i}\in M_{2}}\sqrt{(d_{v}+t)^{2}+d_{v_{i}}^{2}}-\sum_{v_{i}\in M_{2}}\sqrt{d_{v}^{2}+d_{v_{i}}^{2}}\\
\quad & +\sum_{v_{i}\in M_{1}}\sqrt{(d_{v}+t)^{2}+d_{v_{i}}^{2}}-\sum_{v_{i}\in M_{1}}\sqrt{d_{u}^{2}+d_{v_{i}}^{2}}\\
\quad & +\left(\sum_{v_{i}\in M}\sqrt{(d_{v}+t)^{2}+d_{v_{i}}^{2}}-\sum_{v_{i}\in M}\sqrt{d_{v}^{2}+d_{v_{i}}^{2}}\right)\\
\quad & -\left(\sum_{v_{i}\in M}\sqrt{d_{u}^{2}+d_{v_{i}}^{2}}-\sum_{v_{i}\in M}\sqrt{(d_{u}-t)^{2}+d_{v_{i}}^{2}}\right)\\
> &  0. \ \ (\text{by Lemma \ref{l2-2} and $d_{u}\leq d_{v}$} ) &
\end{align*}

This completes the proof.
\end{proof}

Recalling taht $G_{1}\vee G_{2}$ is the graph obtained from $G_{1}\cup G_{2}$ by adding edges between any vertex of $G_{1}$ and any vertex of $G_{2}$.
\begin{theorem}\label{l2-5}
Suppose that $G(i,n-k-i)=K_{i}\vee H_{k}\vee K_{n-k-i}$ is the graph with $n$ vertices and $H_{k}$ is a graph with $k\geq 1$ vertices. Then $SO(G(i,n-k-i))<SO(G(1,n-k-1))$ for $2\leq i\leq \frac{n-k}{2}$.
\end{theorem}
\begin{proof}
Suppose that $V(K_{i})=\{u_{1},u_{2},\cdots,u_{i}\}$, $V(K_{n-k-i})=\{v_{1},u_{2},\cdots,u_{n-k-i}\}$.
Since $2\leq i\leq \frac{n-k}{2}$, then $d_{u_{1}}\leq d_{v_{1}}$. By Lemma \ref{l2-4}, we have
$SO(G_{1})>SO(G(i,n-k-i))$, where $G_{1}=G-\{u_{1}u_{2},u_{1}u_{3},\cdots,u_{1}u_{i}\}+\{v_{1}u_{2},v_{1}u_{3},\cdots,v_{1}u_{i}\}$.
We can also know that $G(1,n-k-1)=G_{1}+\{v_{2}u_{2},v_{2}u_{3},\cdots,v_{2}u_{i}\}+\{v_{3}u_{2},v_{3}u_{3},\cdots,v_{3}u_{i}\}
+\cdots+\{v_{n-k-i}u_{2},v_{n-k-i}u_{3},\cdots,v_{n-k-i}u_{i}\}$. Then by Lemma \ref{l2-3}, we have
$SO(G_{1})<SO(G(1,n-k-1))$. Thus $SO(G(i,n-k-i))<SO(G(1,n-k-1))$ for $2\leq i\leq \frac{n-k}{2}$.

This completes the proof.
\end{proof}

\begin{theorem}\label{t2-6}
Let $G\in \mathbb{V}_{n}^{k}$. Then we have
$SO(G)\leq k\sqrt{k^{2}+(n-1)^{2}}+k(n-k-1)\sqrt{(n-1)^{2}+(n-2)^{2}}+\frac{\sqrt{2}}{2}k(k-1)(n-1)
+\frac{\sqrt{2}}{2}(n-k-1)(n-k-2)(n-2)$,
with equality if and only if $G\cong K_{n}^{k}$.
\end{theorem}
\begin{proof}
If $k=n-1$, then $K_{n}^{n-1}\cong K_{n}\in \mathbb{V}_{n}^{n-1}$, thus the $SO(G)\leq SO(K_{n}^{k})$ for $k= n-1$. Next we only consider the case $1\leq k\leq n-2$.

Suppose that $SO(G)\leq SO(G_{1})$ for any $G\in \mathbb{V}_{n}^{k}$ ($1\leq k\leq n-2$). Thus $G_{1}\ncong K_{n}$, then $G_{1}$ has a $k$-vertex cut that is $W=\{u_{1},u_{2},\cdots,u_{k}\}$.
Denote by $\omega(G)$ the number of connected components in graph $G$.

We first proof that $\omega(G_{1}-W)=2$. Otherwise, $\omega(G_{1}-W)\geq 3$. We suppose that $G_{1}-W=\{H_{1},H_{2},\cdots,H_{l}\}$ $(l\geq 3)$.
Let $u\in H_{i}$ and $v\in H_{j}$ where $i,j\in \{1,2,\cdots,l\}$ and $i\neq j$.
We find that $W$ is still a $k$-vertex cut of $G_{1}+uv$, i.e., $G_{1}+uv\in \mathbb{V}_{n}^{k}$.
Since $SO(G_{1}+uv)>SO(G_{1})$, which is a contradiction with that $SO(G)\leq SO(G_{1})$ for any $G\in \mathbb{V}_{n}^{k}$. Thus we have $\omega(G_{1}-W)=2$.

We suppose that $G_{1}-W=\{H_{1},H_{2}\}$.

Next we proof that $G_{1}[V(H_{1})\cup W]$ and $G_{1}[V(H_{2})\cup W]$ are all cliques.
Otherwise, $G_{1}[V(H_{1})\cup W]$ is not a clique. Then there exist vertices $u,v\in V(H_{1})\cup W$, such that $uv\notin E(G_{1}[V(H_{1})\cup W])$. Since $SO(G_{1}+uv)>SO(G_{1})$ and $G_{1}+uv\in \mathbb{V}_{n}^{k}$, which is a contradiction with that $SO(G)\leq SO(G_{1})$ for any $G\in \mathbb{V}_{n}^{k}$. Thus $G_{1}[V(H_{1})\cup W]$ is a clique. Similarly, we also have that $G_{1}[V(H_{2})\cup W]$ is a clique.

Since $G_{1}[V(H_{1})\cup W]$ and $G_{1}[V(H_{2})\cup W]$ are all cliques, then $H_{1}$ and $H_{2}$ are all cliques. For convenience, we denote $H_{1}$ (resp. $H_{2}$) as $K_{n_{1}}$ (resp. $K_{n_{2}})$.

At last we proof that $n_{1}=1$ or $n_{2}=1$. Otherwise, we have $n_{1}\geq 2$ and $n_{2}\geq 2$.
Suppose that $G_{2}=K_{1}\vee G_{1}[W]\vee K_{n-k-1}$, it is obvious that $G_{2}\in \mathbb{V}_{n}^{k}$. By the conclusion of Lemma \ref{l2-5}, we know that $SO(G_{1})<SO(G_{2})$, which is a contradiction with that $SO(G)\leq SO(G_{1})$ for any $G\in \mathbb{V}_{n}^{k}$. Thus $n_{1}=1$ or $n_{2}=1$.

By a simple calculation, we know that
$SO(G)\leq k\sqrt{k^{2}+(n-1)^{2}}+k(n-k-1)\sqrt{(n-1)^{2}+(n-2)^{2}}+\frac{\sqrt{2}}{2}k(k-1)(n-1)
+\frac{\sqrt{2}}{2}(n-k-1)(n-k-2)(n-2)$,
with equality if and only if $G\cong K_{n}^{k}$.

This completes the proof.
\end{proof}

\begin{corollary}\label{c2-7}
Let $G\in \mathbb{E}_{n}^{k}$. Then we have
$SO(G)\leq k\sqrt{k^{2}+(n-1)^{2}}+k(n-k-1)\sqrt{(n-1)^{2}+(n-2)^{2}}+\frac{\sqrt{2}}{2}k(k-1)(n-1)
+\frac{\sqrt{2}}{2}(n-k-1)(n-k-2)(n-2)$,
with equality if and only if $G\cong K_{n}^{k}$.
\end{corollary}
\begin{proof}
Since $K_{n}^{k}\in \mathbb{E}_{n}^{k}\subseteq \mathbb{V}_{n}^{k}$, then by Theorem \ref{t2-6}, the conclusion holds.
\end{proof}

\begin{theorem}\label{t2-8}
Let $G\in \mathbb{V}_{n}^{k}$ ($k\geq 1$) be a connected graph. Then we have
$SO(G)\geq 2\sqrt{2}(n-3)+2\sqrt{5}$,
with equality if and only if $G\cong P_{n}$.
\end{theorem}
\begin{proof}
Note that for $e=uv\in E(G)$ and $G\in \mathbb{V}_{n}^{k}$, then $G-uv\in \mathbb{V}_{n}^{k}$.
Thus if $G\in \mathbb{V}_{n}^{k}$ be a connected graph with minimum Sombor index, then $G$ must be a trees. Then by Lemma \ref{l2-1}, we have $P_{n}$ is the minimum connected graph in $\mathbb{V}_{n}^{k}$ with respect to Sombor index.
\end{proof}

\begin{corollary}\label{c2-9}
Let $G\in \mathbb{E}_{n}^{k}$ ($k\geq 1$) be a connected graph. Then we have
$SO(G)\geq 2\sqrt{2}(n-3)+2\sqrt{5}$,
with equality if and only if $G\cong P_{n}$.
\end{corollary}
\begin{proof}
Since $P_{n}\in \mathbb{E}_{n}^{k}\subseteq \mathbb{V}_{n}^{k}$, then by Theorem \ref{t2-8}, the conclusion holds.
\end{proof}

\section{Counterexamples}
\hskip 0.6cm
In one recently published paper \cite{hare2023}, there are some flaws in their Lemmas. We give some counterexamples about their results.

\begin{figure}[ht!]
  \centering
  \scalebox{.16}[.16]{\includegraphics{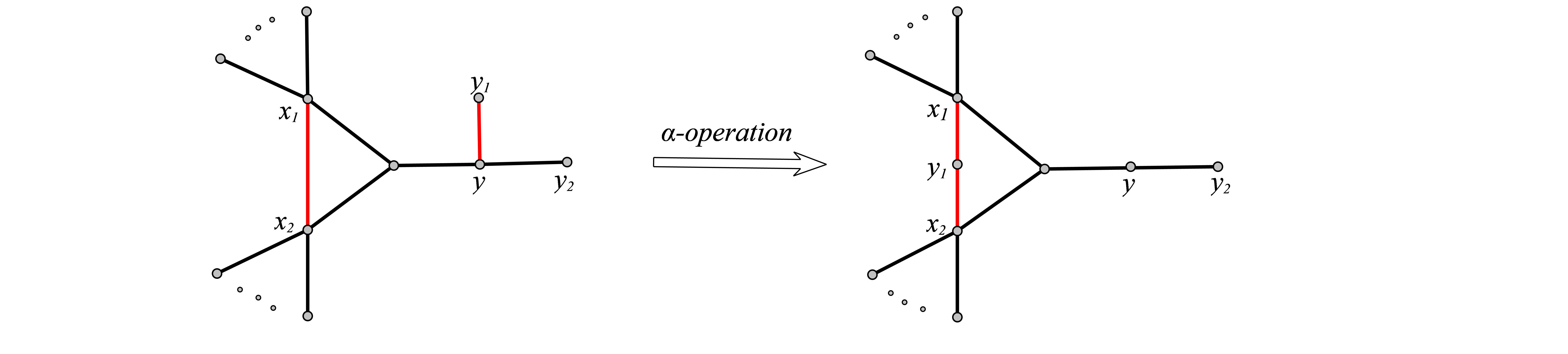}}
  \caption{A counterexamples of Lemma 3 of \cite{hare2023}.}
 \label{fig-31}
\end{figure}

\begin{lemma}\label{l3-1}{\rm\cite{hare2023}\rm}
Let $\Gamma_{\alpha}$ be the $\alpha$-switched graph of a graph $\Gamma$ with $n$ vertices and $k$ cut vertices (see Figure 1 of \cite{hare2023}). Then $SO(\Gamma_{\alpha})<SO(\Gamma)$.
\end{lemma}
The conclusion of Lemma \ref{l3-1} of \cite{hare2023} is not entirely correct. A counterexamples of Lemma 3 of \cite{hare2023} can see Figure \ref{fig-31}. If $d_{x_{1}}= 8$ and $d_{y_{1}}= 8$, then
$SO(\Gamma)\approx 135.716<136.169\approx SO(\Gamma_{\alpha})$, which is a contradiction with the conclusion of Lemma \ref{l3-1} of \cite{hare2023}.
In fact, if $d_{x_{1}}\geq 8$ and $d_{y_{1}}\geq 8$, then the conclusion of Lemma \ref{l3-1} of \cite{hare2023} is wrong.

Let $x,y\geq 2$ and $f(x,y)=\sqrt{x^{2}+y^{2}}-\sqrt{x^{2}+2^{2}}-\sqrt{y^{2}+2^{2}}$.
Then $\frac{\partial f}{\partial x}=\frac{x}{\sqrt{x^{2}+y^{2}}}-\frac{x}{\sqrt{x^{2}+2^{2}}}<0$.
Thus $f(x,y)$ is is monotonically decreasing on $x$ or $y$.
However, in the proof of Lemma \ref{l3-1} of \cite{hare2023}, the authors think that $f(x,y)$ is is monotonically increasing on $x$ or $y$.

The conclusions of Lemma 4 and Lemma 5 of \cite{hare2023} is also not entirely correct. We can also find counterexamples. The main problem is that the author misunderstood the monotonicity of the above function $f(x,y)$.

\section{QSPR analysis with regression modeling and Sombor index}
\hskip 0.6cm
The experimental values of enthalpie of combustion (resp. enthalpie of formation of liquid , enthalpie of sublimation, enthalpie of vaporization) of $19$ monocarboxylic acids of Table 1 are taken from \cite{shai2015}.

The regression modeling between Sombor index and enthalpie of combustion ($\Delta H_{c}^{o}$), enthalpie of formation of liquid ($\Delta H_{f}^{o}$), enthalpie of sublimation ($\Delta H_{sub}^{o}$), enthalpie of vaporization ($\Delta H_{vap}^{o}$) of $19$ monocarboxylic acids are, respectivily,

$$\Delta H_{c}^{o}= 229.7\times SO(G)-1263,\quad  \Delta H_{f}^{o}= 10.65\times SO(G)+369.2,$$
$$\Delta H_{sub}^{o}= 1.212\times SO(G)+36.41,\quad  \Delta H_{vap}^{o}= 2.559\times SO(G)+21.83.$$

By comparison, we find that the Sombor index exerts a better predictive capability than other vertex degree-based indices.

\begin{table}[h]
	\centering
    \caption{Experimental values of enthalpie of combustion ($\Delta H_{c}^{o} KJ/mol$), enthalpie of formation of liquid ($\Delta H_{f}^{o} KJ/mol$), enthalpie of sublimation ($\Delta H_{sub}^{o} KJ/mol$), enthalpie of vaporization ($\Delta H_{vap}^{o} KJ/mol$) and Sombor index of $19$ monocarboxylic acids}
    \setlength{\tabcolsep}{0.9mm}{
	\begin{tabular}{c|c|c|c|c|c}\hline
  Compounds     &  $\Delta H_{c}^{o} KJ/mol$  & $\Delta H_{f}^{o} KJ/mol$  &	$\Delta H_{sub}^{o} KJ/mol$  &	$\Delta H_{vap}^{o} KJ/mol$ &	$SO$    \\  \hline

	$Acetic acid$    &  $875.16$  &    $483.5$  &    $46.3$  &    $49.7$  &    $9.48683$  \\  \hline

    $Propanoic acid$    &  $1527.3$  &    $510.8$  &  $50.0$  &    $56.1$  &    $12.1662$  \\ \hline

    $Butanoic acid$    &  $2183.5$  &    $533.9$  &	  $54.9$  &    $62.9$  &    $14.9946$  \\ \hline

    $Pentanoic acid$    &  $2837.8$  &    $558.9$  &  $58.2$  &    $69.0$  &    $17.8230$  \\ \hline

    $Hexanoic acid$    &  $3494.3$  &    $581.8$  &	  $63.0$  &    $75.0$  &    $20.6515$   \\ \hline

    $Heptanoic acid$    &  $4146.9$  &    $608.5$  &  $64.8$  &    $81.7$  &    $23.4799$  \\ \hline

    $Octanoic acid$    &  $4799.9$  &    $634.8$  &	  $69.4$  &    $86.9$  &    $26.3083$  \\ \hline

    $Nonanoic acid$    &  $5456.1$  &    $658.0$  &	  $72.3$  &    $93.6$  &    $29.1367$  \\ \hline

    $Decanoic acid$    &  $6079.3$  &    $713.7$  &	  $76.3$  &    $100.8$  &    $31.9652$  \\ \hline

    $Undecanoic acid$   &  $6736.5$  &    $736.2$  &  $78.9$  &    $106.7$  &    $34.7936$  \\ \hline

    $Dodecanoic acid$   &  $7333.0$  &    $775.1$  &  $82.2$  &    $115.9$  &    $37.6220$   \\ \hline

    $Tridecanoic acid$   &  $8024.2$  &    $807.2$  &  $84.9$  &    $121.2$  &    $40.4504$  \\ \hline

    $Tetradecanoic acid$   &  $8676.7$  &    $834.1$  &	  $87.7$  &    $130.2$  &    $43.2789$  \\ \hline

    $Pentadecanoic acid$   &  $9327.7$  &    $862.4$  &	  $91.4$  &    $136.5$  &    $46.1073$  \\ \hline

    $Hexadecanoic acid$   &  $9977.2$  &    $892.2$  &	  $94.5$  &    $144.3$  &    $48.9357$  \\ \hline

    $Heptadecanoic acid$   &  $10624.4$  &    $924.4$  &  $100.7$  &    $159.6$  &    $51.7642$  \\ \hline

    $Octadecanoic acid$   &  $11280.1$  &    $947.2$  &	  $102.8$  &    $164.7$  &    $54.5926$  \\ \hline

    $Nonadecanoic acid$   &  $11923.4$  &    $984.1$  &	  $105.0$  &    $172.9$  &    $57.4210$  \\ \hline

    $Eicosanoic acid$   &  $12574.2$  &    $1012.6$  &	  $109.9$  &    $179.2$  &    $60.2494$  \\ \hline

	\end{tabular}}
	
	\label{table1}
\end{table}

\begin{table}[!htb]
	\centering
    \caption{The $R^{2}$ and $RMSE$ of regression modeling between Sombor index and enthalpie of combustion ($\Delta H_{c}^{o} KJ/mol$), enthalpie of formation of liquid ($\Delta H_{f}^{o} KJ/mol$), enthalpie of sublimation ($\Delta H_{sub}^{o} KJ/mol$), enthalpie of vaporization ($\Delta H_{vap}^{o} KJ/mol$) of $19$ monocarboxylic acids.}
     \setlength{\tabcolsep}{0.7mm}{
	\begin{tabular}{c|ccc}\hline
	Physico-chemical property           &	   $R^{2}$  &      $RMSE$          \\ \hline

	enthalpie of combustion             &     $0.99998$    &	  $17.987$     \\ \hline

    enthalpie of formation of liquid    &     $0.99737$    &	  $8.9567$     \\ \hline

    enthalpie of vaporization           &     $0.99745$    &	  $1.0034$     \\ \hline

    enthalpie of vaporization           &     $0.99355$    &	  $3.2771$     \\ \hline
	\end{tabular}}
	
	\label{table2}
\end{table}

\end{document}